\documentclass[a4paper,10pt]{article}
\usepackage[margin=20mm,nohead]{geometry}

\usepackage{latexsym,amssymb,graphics,graphicx,epsfig,color,enumerate}
\usepackage{amsmath,amsthm}

\usepackage{t1enc}
\usepackage[latin2]{inputenc}

\makeatletter

\newcounter{Kulso}

\newenvironment{pszkod}[2][1]
{
  \setcounter{Kulso}{#1}
  \sffamily
  \vspace{\topsep}              
  \noindent
  {#2}
  \begin{list}{\rmfamily\arabic{Kulso}.\@arabic\c@enumiv.}
    {\usecounter{enumiv}%
     \renewcommand\p@enumiv{\arabic{Kulso}.}
     \renewcommand\theenumiv{\@arabic\c@enumiv}
     \setlength{\parsep}{0pt}
     \setlength{\itemsep}{0pt}
     \setlength{\topsep}{0pt}
    }
  \psz@cimke{begin}
}
{
  \psz@cimke{end}
  \end{list}
  \vspace{\topsep}
}

\newlength{\tabhossz}
\newcommand{\tab}{\par
\advance\@totalleftmargin \tabhossz
\advance\linewidth -\tabhossz
\parshape \@ne \@totalleftmargin \linewidth
\addtolength{\labelsep}{\tabhossz}%
}
\newcommand{\untab}{\par
\advance\@totalleftmargin -\tabhossz   
\advance\linewidth \tabhossz
\parshape \@ne \@totalleftmargin \linewidth
\addtolength{\labelsep}{-\tabhossz}%
}

\newcommand{\psz@cimke}[1]{\item[{\makebox[\labelwidth][l]{#1}}]}

\makeatother

\newcommand{\cL}{\ensuremath{\mathcal{L}}}
\newcommand{\cB}{\ensuremath{\mathcal{B}}}
\newcommand{\cE}{\ensuremath{\mathcal{E}}}
\newcommand{\cI}{\ensuremath{\mathcal{I}}}
\newcommand{\cF}{\ensuremath{\mathcal{F}}}
\newcommand{\cX}{\ensuremath{\mathcal{X}}}
\newcommand{\tD}{\tilde D}

\newcommand{\Rset}{\mathbb{R}}
\newcommand{\Zset}{\mathbb{Z}}

\newtheorem{claimnum}{Claim}

\newtheorem{theorem}{Theorem}
\newtheorem{problem}{Problem}
\newtheorem{lemma}{Lemma}
\newtheorem{definition}{Definition}
\newtheorem{cor}{Corollary}

\title{Blocking optimal arborescences}
\author{Attila Bern\'ath\thanks{MTA-ELTE Egerv\'ary Research Group,
Department of Operations Research, E\"otv\"os University, P\'azm\'any P\'eter s\'et\'any 1/C, Budapest, Hungary, H-1117.
Research supported by the Hungarian Scientific Research Fund (OTKA, grant number K109240), and by the ERC StG project PAAl  no.\ 259515.  Part of the research was done while the author was at Warsaw
University, Institute of Informatics, ul. Banacha 2, 02-097 Warsaw,
Poland. 
E-mail: {\tt bernath@cs.elte.hu}.} 
\and 
Gyula Pap\thanks{MTA-ELTE Egerv\'ary Research Group,
Department of Operations Research, E\"otv\"os University, P\'azm\'any P\'eter s\'et\'any 1/C, Budapest, Hungary, H-1117. Supported by  the Hungarian Scientific Research Fund (OTKA, grant number K109240).
E-mail: {\tt gyuszko@cs.elte.hu}.} 
}

\begin{document}
\maketitle


\begin{abstract}
The problem of covering minimum cost common bases of two matroids is
NP-complete, even if the two matroids coincide, and the costs are all
equal to 1. In this paper we show that the following special case is
solvable in polynomial time: given a digraph $D=(V,A)$ with a
designated root node $r\in V$ and arc-costs $c:A\to \Rset$, find a
minimum cardinality subset $H$ of the arc set $A$ such that $H$
intersects every minimum $c$-cost $r$-arborescence. By an
$r$-arborescence we mean a spanning arborescence of root $r$. The
algorithm we give solves a weighted version as well, in which a
nonnegative weight function $w:A\to \Rset_+$ (unrelated to $c$) is
also given, and we want to find a subset $H$ of the arc set such that
$H$ intersects every minimum $c$-cost $r$-arborescence, and
$w(H)=\sum_{a\in H}w(a)$ is minimum. The running time of the algorithm
is $O(n^3T(n,m))$, where $n$ and $m$ denote the number of nodes and
arcs of the input digraph, and $T(n,m)$ is the time needed for a
minimum $s-t$ cut computation in this digraph. A polyhedral
description is not given, and seems rather challenging.
\end{abstract}

\begin{quote}
\end{quote}
\vspace{5mm}



\section{Introduction}

Let $D=(V,A)$ be a digraph with vertex set $V$ and arc set $A$. A
\textbf{spanning arborescence} is a subset $B\subseteq A$ that is a
spanning tree in the undirected sense, and every node has in-degree at
most one. Thus there is exactly one node with in-degree zero, we call
it the \textbf{root} node.  Equivalently, a spanning arborescence is a
subset $B\subseteq A$ with the property that there is a root node
$r\in V$ such that $\varrho_B(r)=0$, and $\varrho_B(v)=1$ for $v\in
V-r$, and $B$ contains no cycle. An \textbf{arborescence} will mean a
spanning arborescence, unless stated otherwise. If $r\in V$ is the
root of the spanning arborescence $B$ then we will say that $B$ is an
\textbf{$r$-arborescence}.  Arborescences form a natural directed
variant of spanning trees and they arise in applications where we want
to broadcast information in a network that contains directed
connections, for example wireless devices, too.

The Minimum Cost Arborescence Problem is the following: given a
digraph $D=(V,A)$, a designated root node $r\in V$ and a cost function
$c:A\to \Rset$, find an $r$-arborescence $B\subseteq A$ such that the
cost $c(B)=\sum_{b\in B}c(b)$ of $B$ is smallest possible. 
Fulkerson \cite{fulk} has given a two-phase algorithm solving this
problem, and he also characterized minimum cost arborescences.  A
natural variant of the Minimum Cost Arborescence Problem is where the
root is not fixed, i.e. find an arborescence in a digraph that has
minimum total cost. A simple reduction
shows that this \textit{global minimum cost arborescence problem} is
equivalent with the Minimum Cost Arborescence Problem stated above.
Kamiyama in \cite{Kamiyama} raised the following question.

\begin{problem}[Blocking Optimal $r$-arborescences]\label{prob:1}
Given a digraph $D=(V,A)$, a designated root node $r\in V$ and a cost
function $c:A\to \Rset$, find a subset $H$ of the arc set such that $H$
intersects every minimum cost $r$-arborescence, and $|H|$ is minimum.
\end{problem}

The minimum in Problem \ref{prob:1} measures the robustness of the
minimum cost arborescences, since it asks to delete a minimum
cardinality set of arcs in order to destroy all minimum cost
$r$-arborescences. One might ask why we fix the root of the
arborescences that we want to cover. The problem of finding the minimum
number of arcs that intersect every (globally) minimum cost
arborescence can be reduced to Problem \ref{prob:1}as follows. Add a new node
$r'$ to the digraph and connect $r'$ with every old node by high cost
and high multiplicity arcs. Then minimum cost arborescences in this
new instance will be necessarily rooted at $r'$, they will only
contain one arc leaving $r'$ by the high cost of these arcs, and an
optimal arc set intersecting these will not use the new arcs because
of their high multiplicity.

One interpretation of Problem \ref{prob:1} is that we want to
\emph{cover the minimum cost common bases of two matroids}: one
matroid being the graphic matroid of $D$ (in the undirected sense),
the other being a partition matroid with partition classes
$\delta^{in}(v)$ for every $v\in V-r$.  The problem of covering
minimum cost common bases of two matroids is in general NP-complete,
even if the two matroids coincide, and the costs are all equal to
1. That is, given a matroid with an independence oracle, it is
NP-complete to find a minimum cut (where a \textbf{cut} in a matroid
is an inclusionwise minimal subset that intersects every base). For a
proof, see for example Theorem 2.1 in \cite{minkor} and apply it to
the dual matroid. In \cite{qp:egres} we investigated the problem of
\emph{covering all minimum cost bases of a matroid}. We have given an
algorithm that solves the problem in polynomial time, provided that we
can find a minimum cut in certain minors of the matroid at hand. This
includes the class of graphic matroids, that is the following problem
is polynomial time solvable: given an undirected graph and a cost
function on its edges, find a minimum size set of edges which
intersects every minimum cost spanning tree. In fact this is a not too
hard exercise for the interested reader, however the details can be
found in \cite{qp:egres}.

 Kamiyama \cite{Kamiyama} solved special cases of  Problem \ref{prob:1} and he
 investigated some necessary and sufficient conditions for the minimum
 in this problem. In this paper we give a polynomial time algorithm
 solving Problem \ref{prob:1}.  In fact, our algorithm will solve the
 following, more general problem, too.

\begin{problem}\label{prob:1'}
Given a digraph $D=(V,A)$, a designated root node $r\in V$, a cost
function $c:A\to \Rset$ and a nonnegative weight function $w:A\to
\Rset_+$, find a subset $H$ of the arc set such that $H$ intersects every
minimum $c$-cost $r$-arborescence, and $w(H)$ is minimum.
\end{problem}


The rest of this paper is organized as follows. In Section
\ref{sec:related} we describe related work. In Section
\ref{sec:variants} we give variants of the problem based on
Fulkerson's characterization of minimum cost arborescences. These
variants are all equivalent with Problem \ref{prob:1} as simple
reductions show, so we deal with the variant that can be handled most
conveniently. In Section \ref{sec:special} we solve the special case
of covering all arborescences. This is indeed a very special case, but
the answer is very useful in the solution of the general case. Section
\ref{sec:covtight} contains our main result broken down into two
steps: in Section \ref{sec:minmin} we introduce a min-min formula that
gives a useful reformulation of our problem, and we give two proofs of
this min-min formula (in Section \ref{sec:minmin} and in Section
\ref{sec:reprove}). Finally --after introducing some essential results
and techniques in Section \ref{sec:solid}-- we give a polynomial time
algorithm solving Problems \ref{prob:1} and \ref{prob:1'} in Section
\ref{sec:algorithm}. We analyze the running time of this algorithm in
Section \ref{sec:runtime}.  In Section \ref{sec:fur} we give further
remarks: we reduce Problems \ref{prob:1} and \ref{prob:1'} to the
problem of covering the common bases of a graphic matroid and a
partition matroid, which might be a direction for further research.

\section{Related work}\label{sec:related}

The problem that we consider (Problem \ref{prob:1} and its weighted
version, Problem \ref{prob:1'}) is a typical \textbf{ covering
  problem}: given a hypergraph on ground set $S$ with hyperedge set $\cE\subseteq 2^S$,
find a minimum size subset $F$ of $S$ that intersects every hyperedge
in $\cE$. In Problem \ref{prob:1} the ground set $S$ is the set of the
arcs of a digraph, and $\cE$ is the set of optimal
$r$-arborescences. Every hypergraph covering problem defines a
\textbf{ packing problem} in a natural way: given a hypergraph with
hyperedge set $\cE\subseteq 2^S$, find a maximum size subset $\cE'$
of pairwise disjoint hyperedges. For our special case of Problem
\ref{prob:1}, this packing problem can be solved as follows: (a) for
any natural number $k$ we can find a minimum cost subset $B_k$ of the
arc set such that $B_k$ is the union of $k$ arc-disjoint
$r$-arborescences, and (b) we can thus find (with a logarithmic
search) the largest $k$ such that the cost of $B_k$ is $k$ times the
cost of an optimal $r$-arborescence. Hypergraph packing and covering
problems have an extensive literature, see for example the surveys
\cite{cornuejols2001combinatorial} or \cite[Part
  VIII.]{Schrijver}. Most of these results deal with certain
properties of hypergraphs (e.g. packing property, MFMC property, ideal
hypergraphs, hypergraphs that pack, etc.) that help solving the
packing and the covering problem. These results are LP based, as they
use an LP relaxation of the packing/covering problem pair.

We looked at each of these general hypergraph covering frameworks in
order to find a candidate to help solve Problems \ref{prob:1} and
\ref{prob:1'}, to no avail. It seems that these problems are outside
of the reach of the models known in the literature.

If the cost function $c$ is uniform then even the weighted
Problem \ref{prob:1'} is easy to solve (as the task is to find a
minimum weight $r$-cut, where an \textbf{$r$-cut} is the set of arcs
entering a non-empty subset $X\subseteq V-r$).  By a result of Edmonds
\cite{edmonds1973edge} (see also \cite[Theorem
  1.24]{cornuejols2001combinatorial}) a polyhedral description can
also be given. Namely the dominant of incidence vectors of $r$-cuts in
a digraph $D=(V,A)$ can be described as $\{x\in \Rset^A: x\ge 0,
x(B)\ge 1 $ for every $r$-arborescence $B\}$.

As we already mentioned in the introduction, Problems \ref{prob:1} and
\ref{prob:1'} are special cases of covering the minimum cost common
bases of two matroids. It is not too hard to show, that given 2
matroids $M_i=(S, \cI_i)$ ($i=1,2$) over the same ground set $S$ and a
cost function $c:S\to \Rset$, the minimum cost common bases of $M_1$
and $M_2$ are the common bases of two other matroids $M_1'$ and $M_2'$
over the same ground set (see details in Section \ref{sec:fur}). That is, Problems \ref{prob:1} and
\ref{prob:1'} fall in the framework of covering the common bases of two
matroids. Another natural special case of covering common bases is the
following. Given a bipartite graph $G=(S,T,E)$, find a minimum size
subset $H\subseteq E$ such that $G-H$ does not contain a perfect
matching. However, this problem was shown to be NP-complete by Joret
and Vetta \cite{joret}.

Another related problem is the minimum cut problem in matroids. A
\textbf{cut} of matroid is an inclusionwise minimal subset that
intersects every base. The minimum cut problem in a matroid is the
problem of finding a minimum size cut.  This problem is also known as
the \textbf{cogirth problem} (the cogirth of a matroid is the size of
the smallest cut, while the girth of a matroid is the size of the
smallest circuit). Since the cuts of a matroid are the circuits of the
dual matroid, the minimum cut problem is equivalent with the minimum
curcuit problem (a.k.a. \textbf{girth problem}). The (co)girth problem
in matroids is NP-hard in general (e.g. girth is NP-hard in a transversal matroid
\cite{mccormick}, and both girth and cogirth are NP-hard in a binary matroid \cite{vardy}).  
Polynomially solvable cases of
the (co)girth problem has received considerable attention. For
example, Geelen et al \cite{geelen2015computing}, motivated by
applications in coding theory, devised a polynomial time randomized
algorithm solving the (co)girth problem in binary matroids that arise
from a graphic matroid by adding a small perturbation. The minimum cut
problem is polynomially solvable in $k$-graphic-matroids (that is,
matroids of the form $kM$ where $M$ is graphic); this was generalized
by Király \cite{egresqp-09-05} for $k$-hypergraphic matroids
(i.e. matroids of the form $kM$ where $M$ is hypergraphic;
hypergraphic matroids were first defined by Lorea
\cite{lorea1975hypergraphes}). 


\textbf{Interdiction
  problems} are also related to our research. In interdiction problems there are two (unrelated)
objective functions like in Problem \ref{prob:1'} (lets call them cost
and weight), and we are interested in minimum cost objects satisfying
some requirement (for example spanning trees, matchings, or shortest
paths). The problem is then to increase the cost of these optimal
objects to the greatest extent by removing elements with some budget
for the total weight of the removed elements. We mention that there
are other types of interdiction problems, too, but these are closest
to our research. An example is the \textbf{$k$ most vital edges
  problem}: given an undirected graph with costs on the edges, and the
task is to delete $k$ edges so that the minimum cost of a spanning
tree in the remaining graph is as large as possible (in this case the
weights are all 1, and $k$ plays the role of the budget). In
\cite{frederickson1999increasing} it is shown that this problem is
NP-complete, and in \cite{liang1997finding} it is shown that it can be
solved in time $O(n^{k+1})$. The \textbf{matching interdiction
  problem} is the following. Given a graph $G$ with non-negative costs
and weights on its edges and a budget, find a subset of edges $R$ such
that the total cost of $R$ does not exceed the budget and the maximum
weight of a matching in $G-R$ is smallest possible (note that
  for this sole problem we switched the role of the cost and the
  weight functions for conventional reasons). In \cite[Theorem 3.3]{zenklusen2009blockers} it
is shown that this problem is NP-complete on simple bipartite graphs
with unit edge weights and unit interdiction costs, that is, it is NP
hard to find a minimum size set of edges $F$ of a bipartite graph $G$
such that $F$ intersects every maximum matching of $G$ (as mentioned
above, in \cite{joret} it is shown that this problem is NP-hard even
for bipartite graphs that admit a perfect matching).  In
\cite{zenklusen2010matching} a 4-approximation algorithm is given for
the special case when all edge weights are 1.  This is improved in
\cite{dinitz2013packing} where it is shown that there exists a
constant factor approximation algorithm for the matching interdiction
problem even without this restriction on the edge weights.

\section{The problem and its variants}\label{sec:variants}

In this paper we investigate Problem \ref{prob:1} and the more general
Problem \ref{prob:1'}.   For sake of simplicity we will mostly speak
about Problem \ref{prob:1}, and in Section \ref{sec:algorithm} we
sketch the necessary modifications of our algorithm needed to solve
Problem \ref{prob:1'}. Note that Problem \ref{prob:1'} with an integer
weight function $w$ can be reduced to Problem \ref{prob:1} by
replacing an arc $a\in A$ (of weight $w(a)$) with $w(a)$ parallel
copies (each of weight 1). This reduction is however not
polynomial. On the other hand, the algorithm we give for Problem
\ref{prob:1} can be simply modified to solve Problem \ref{prob:1'} in
strongly polynomial time.

Let us give some more definitions.  The arc set of the digraph $D$
will also be denoted by $A(D)$.  Given a digraph $D=(V,A)$ and a node
set $Z\subseteq V$, let $D[Z]$ be the digraph obtained from $D$ by
deleting the nodes of $V-Z$ (and all the arcs incident with them). If
$B\subseteq A$ is a subset of the arc set, then we will sometimes
abuse the notation by identifying $B$ and the graph $(V,B)$: thus
$B[Z]$ is obtained from $(V,B)$ by deleting the nodes of $V-Z$ (and
the arcs of $B$ incident with them). The set of arcs of $D$ entering
(leaving) $Z$ is denoted $\delta_D^{in}(Z)$ ($\delta_D^{out}(Z)$), the
number of these arcs is $\varrho_D(Z)=|\delta_D^{in}(Z)|$
($\delta_D(Z)=|\delta_D^{out}(Z)|$, respectively). We will omit $D$
from the subscript in these notations, if no confusion can arise.

The following theorem of Fulkerson characterizes the minimum cost
arborescences and leads us to a more convenient, but equivalent
problem.

\begin{theorem}[Fulkerson,  \cite{fulk}]\label{thm:fulk}
There exists a subset $A'\subseteq A$ of arcs (called \emph{tight
  arcs}) and a laminar family $\cL\subseteq 2^{V-r}$ such that an
$r$-arborescence is of minimum cost if and only if it uses only tight
arcs and it enters every member of $\cL$ exactly once. The set $A'$
and the family $\cL$ can be found in polynomial time.
\end{theorem}

Since non-tight arcs do not play a role in our problems, we can forget
about them, so we assume that $A'=A$ from now on.

Let $\cL$ be a laminar family of subsets of $V$.  A spanning
arborescence $B\subseteq A$ in $D$ is called a \textbf{$\cL$-tight
  arborescence} if both of the following hold.
\begin{enumerate}
 \item
$|\delta ^{in} _B(F)|\le 1$ for all $F\in \cL$, and 
 \item
$|\delta ^{in} _B(F)|=0$ for all $F\in \cL$ containing the root $r$ of $B$. 
\end{enumerate}

We point out that the second condition in the above definition is
needed because we don't want to fix the root of the arborescences:
this will be natural in the solution we give for Problem
\ref{prob:1}. The following statement gives an equivalent definition
of $\cL$-tight arborescences; the proof is left to the reader.

\begin{claimnum}\label{cl:Ltight}
A spanning
arborescence $B\subseteq A$ in $D$ is  \emph{$\cL$-tight} if and only if $B[F]$ is an  arborescence for every $F\in \cL$.
\end{claimnum}

 The result of Fulkerson leads us to the
following problem.

\begin{problem}\label{prob:2}
Given a digraph $D=(V,A)$, a designated root node $r\in V$ and a
laminar family $\cL\subseteq 2^V$, find a subset $H$ of the arc set such
that $H$ intersects every $r$-rooted $\cL$-tight arborescence and $|H|$ is
minimum.
\end{problem}

Note that in this problem we allow that $r\in F$ for some members
$F\in \cL$. By Fulkerson's Theorem above, if we have a polynomial
algorithm for Problem \ref{prob:2} then we can also solve Problem
\ref{prob:1} in polynomial time with this algorithm. However, this can
be reversed by the next claim. 

\begin{claimnum}\label{cl:1ekv2}
If we have a polynomial algorithm solving Problem \ref{prob:1} then we
can also solve Problem $\ref{prob:2}$ in polynomial time. 
\end{claimnum}

\begin{proof}
Let the cost of an arc $a\in A$ be equal to the number of sets 
$F\in\cL$ such that $a$ enters $F$.  Then an $r$-arborescence is of
minimum cost if and only if it is $\cL$-tight, if there exists an
$\cL$-tight arborescence at all.
\end{proof}

We point out that the construction in the above proof also shows how
to find an $\cL$-tight arborescence, if it exists at all.  So we can
turn our attention to Problem $\ref{prob:2}$.  However, in order to
have a more compact answer, it is more convenient to consider the
following, equivalent problem instead, in which the root is not
designated. 

\begin{problem}\label{prob:3}
Given a digraph $D=(V,A)$ and a
laminar family $\cL\subseteq 2^V$, find a subset $H$ of the arc set such
that $H$ intersects every $\cL$-tight arborescence and $|H|$ is
minimum.
\end{problem}

\begin{claimnum}\label{cl:2ekv3}
There exists a polynomial algorithm solving Problem \ref{prob:2} if
and only if there exists a polynomial algorithm solving Problem
\ref{prob:3}.
\end{claimnum}

\begin{proof}
Assume that there exists a polynomial time algorithm for Problem
\ref{prob:3} and consider an instance of Problem \ref{prob:2}. Since
arcs entering $r$ will not be used in an optimal solution of Problem
\ref{prob:2}, we can assume that there are no such arcs. But then the
$\cL$-tight arborescences are all rooted in $r$, so our algorithm covering all $\cL$-tight arborescences can be used for solving Problem \ref{prob:2}, too.

For the other direction assume that we have a polynomial time algorithm
solving Problem \ref{prob:2} and consider an instance of Problem
\ref{prob:3} given with $D=(V,A)$ and $\cL\subseteq 2^V$. Let
$V'=V+r'$ with a new node $r'$ and let $D'=(V',A')$ where $A'$
consists of the arcs in $A$ and an  arc of multiplicity $|A|+1$ from $r'$ to
every $v\in V$. Finally let $\cL'=\cL+\{V\}$. Now consider Problem
\ref{prob:2} with input $D', r'$ and $\cL'$.  Observe that $\cL$-tight
arborescences in $D$ and ($r'$-rooted) $\cL'$-tight arborescences in
$D'$ correspond to each-other in a natural way, and an optimal
solution to this instance of Problem \ref{prob:2} will not contain any
arc of form $r'v$ (if it contains one then it has to contain all
parallel copies, but we included those with a large
multiplicity).
\end{proof}

The main result of this paper is a polynomial algorithm solving
Problem \ref{prob:3}, and thus, by Claims \ref{cl:1ekv2} and
\ref{cl:2ekv3}, for Problems \ref{prob:1} and \ref{prob:2}.  For a
digraph $D=(V,A)$, and a laminar family $\cL$ of subsets of $V$, let
$\gamma (D,\cL)$ denote the minimum number of arcs deleted from $D$ to
obtain a digraph that does not contain an $\cL$-tight arborescence,
that is,
\begin{equation}\label{min1}
\gamma (D,\cL):=\min \{|H|: H\subseteq A \text{ such that }
D-H  \text{contains no $\cL$-tight arborescence }\}.
\end{equation}

Notice that an arborescence is \cL-tight if and only if it is
($\cL\cup\{V\}$)-tight, and $\gamma (D,\cL) = \gamma (D,\cL\cup\{V\})$. Therefore we will assume that $V\in \cL$ from now on.

\section{Covering  all arborescences -- a special case}\label{sec:special}

In the proof of our main result below, we will use its
special case when the laminar family $\cL=\{V\}$. This
special case amounts to the following well-known characterization of
the existence of a spanning arborescence.

\begin{lemma}\label{lem:arb}
For any digraph $D=(V,A)$ exactly one of the
following two alternatives holds:
\begin{enumerate}
 \item 
there exists a spanning arborescence,
\item 
there exist two disjoint non-empty subsets $Z_1,Z_2\subset V$ such that $\varrho_D(Z_1)=\varrho_D(Z_2)=0$.
\end{enumerate}
\end{lemma}

This characterization also implies a formula to determine the minimum
number of edges to be deleted to destroy all arborescences. The
characterization is based on double cuts.
\begin{definition}
For a digraph $D=(V,A)$, a \emph{double cut} $\delta
^{in}(Z_1)\cup\delta ^{in}(Z_2)$ is determined by a pair of non-empty
disjoint node subsets $Z_1,Z_2\subseteq V$. The minimum cardinality of
a double cut is denoted by $\mu(D)$, that is
\begin{equation}
\mu (D):= \min \{|\delta ^{in}(Z_1)|+|\delta ^{in}(Z_2)|: Z_1\cap Z_2=\emptyset \neq Z_1,Z_2\subsetneq V\}.
\end{equation}
\end{definition}


\begin{cor}\label{cor:noarb}
For any digraph $D=(V,A)$ the following
equation holds: $\gamma (D,\emptyset)=\mu (D)$.
\end{cor}

We point out that a minimum double cut can be found in polynomial time
by a simple reduction to minimum cut: this is described in Section \ref{sec:runtime}.  Furthermore we will need the
following observation (the proof is left to the reader).

\begin{lemma}\label{lem:roots}
Given a digraph $D=(V,A)$, let $R=\{r\in V:$ there exists an $r$-rooted
spanning arborescence in $D\}$. Then $D[R]$ is a strongly connected
digraph, and $\varrho_D(R)=0$.
\end{lemma}

\section{Covering  tight arborescences}\label{sec:covtight}

Given a laminar family $\cL\subseteq 2^V$ with $V\in \cL$, for $F\in
\cL$, let $\cL[F]:=\{F'\in\cL,F'\subseteq F\}$.  A simple corollary of
Claim \ref{cl:Ltight} is the following.


\begin{claimnum} \label{haf}
For any $\cL$-tight arborescence $B$, and any
$F\in\cL$, $B[F]$ is an $\cL[F]$-tight
arborescence in $D[F]$.
\end{claimnum}


The following observation is crucial in our proofs.  Given a digraph
$D=(V,A)$ and a laminar family $\cL\subseteq 2^V$, for an arbitrary
member $F\in \cL$ and arc $a=xy\in A$ \emph{leaving} $F$, let $\tD$ be
the graph obtained from $D$ by changing the tail of $a$ for an
arbitrary other node $x'\in F$, that is $\tD=D-xy+x'y$ (where $x,x'\in
F$ and $y\notin F$). This operation will be called a
\textbf{tail-relocation}.  Clearly, there is a natural bijection
between the arcs of $D$ and those of $\tD$, but even more importantly,
this bijection also induces a bijection between the $\cL$-tight
arborescences in $D$ and those in $\tD$. This is formulated in the
following claim.
\begin{claimnum}\label{cl:tailreloc}
Let $B\subseteq A$  and $xy\in B$.
Then $B-xy+x'y$ is an \cL-tight arborescence in $\tD$ if and only if $B$ is an $\cL$-tight arborescence in $D$.
\end{claimnum}

The claim also implies that $\gamma (D,\cL)=\gamma (\tD,\cL)$.

\subsection{A "min-min" formula}\label{sec:minmin}

Our approach to determine $\gamma(D,\cL)$ is broken down
into two steps. First, we prove a "min-min" formula, that is, we show
that a set $H$ that attains the minimum in \eqref{min1} is equal to a
special arc subset called an 
\emph{$\cL$-double-cut}. The second step will be the
construction of an algorithm to find a minimum cardinality
$\cL$-double-cut.

So what is this first step -- the min-min formula all about? It
expresses that in order to cover optimally the $\cL$-tight
arborescences we need to consider the problem of covering the $\cL[F]$-tight arborescences for every $F\in \cL$.


\begin{definition}
For a set $Z\subseteq V$, let $\cL_Z$ denote the family of sets in $\cL$ not disjoint from $Z$, that is, let
\begin{equation}
\cL_Z:=\{F\in \cL:F\cap Z\ne\emptyset\}.
\end{equation}
Then an \emph{$\cL$-cut} $M(Z)$ is defined as the set of arcs entering $Z$, but not leaving any set in $\cL_Z$, that is, let
\begin{equation}
M(Z):=M_{D,\cL}(Z):=
\delta _D^{in}(Z)-\bigcup _{F\in \cL_Z}\left( \delta ^{out} _D(F) \right) .
\end{equation}
\end{definition}
\begin{figure}[!ht]
\begin{center}
\includegraphics[width=8cm]{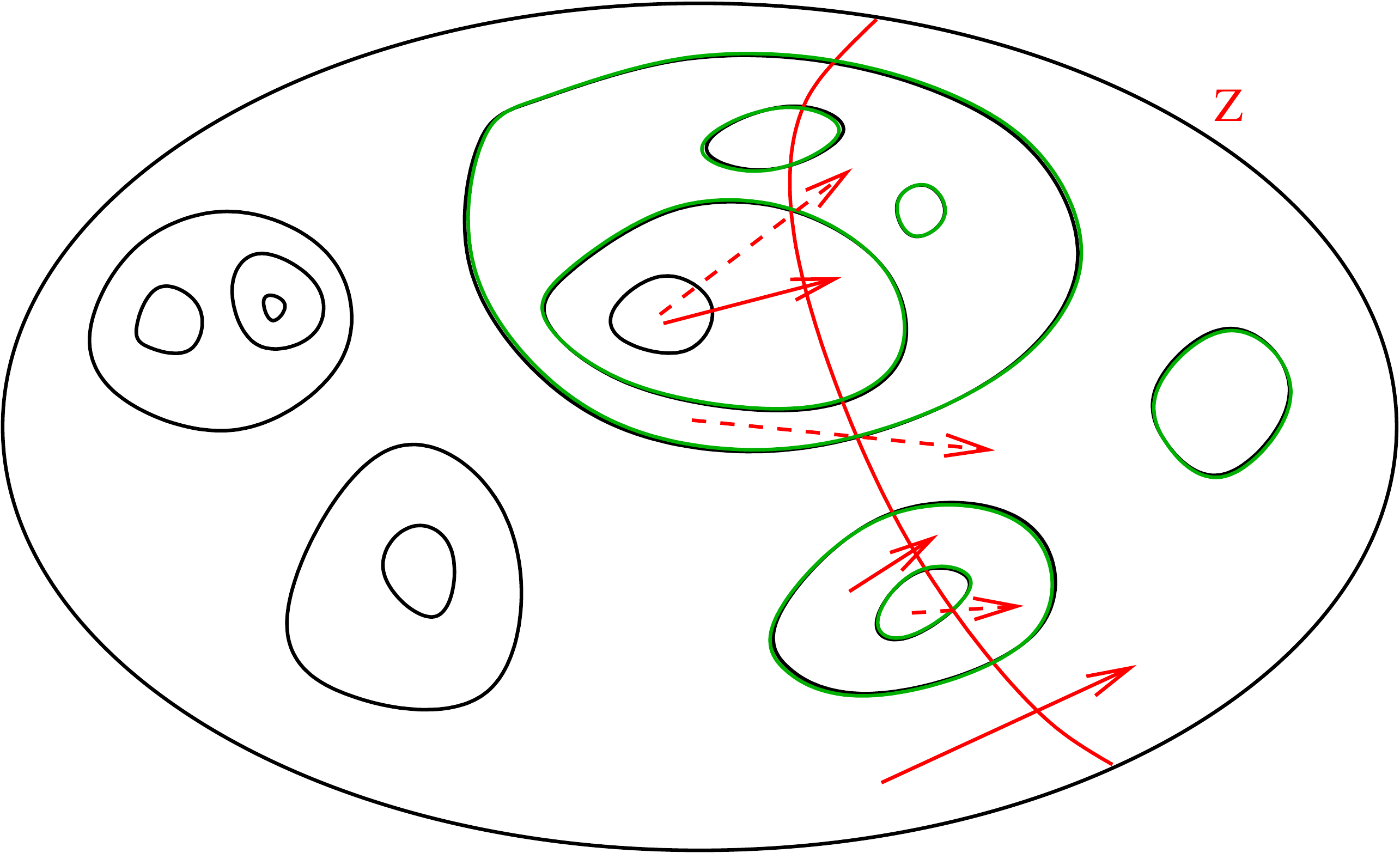}
\caption{Illustration of an $\cL$-cut. The arcs drawn with a solid line count in $M(Z)$, those with dashed line do not.}
\label{fig:lcut}
\end{center}
\end{figure}

Thus $M(Z)$ consists of those arcs
entering $Z$, but not leaving any of those sets in $\cL$ that have
non-empty intersection with $Z$. An illustration can be found in Figure \ref{fig:lcut}.  A set function $f$ is given by the
cardinality of an $\cL$-cut, that is, we define
\begin{equation}
f(Z):=f_D(Z):=f_{D,\cL}(Z):=|M_{D,\cL}(Z)|.
\end{equation}
It is useful to observe that 
\begin{equation}
f_{D,\cL}(Z)\ge f_{D[F],\cL[F]}(Z\cap F)\mbox{ for any }F\in \cL.
\end{equation}
The motivation for $f$ and $M(Z)$ is that $H=M(Z)$ is a set of arcs the deletion of which destroys all tight arborescences rooted outside of $Z$, as claimed by the following lemma.

\begin{lemma}\label{lem:norout}
Given $D=(V,A)$ and laminar family $\cL\subseteq 2^V$.  If $f_D(Z)=0$
for some non-empty $Z\subseteq V$ then there is no \cL-tight
arborescence in $D$ with root outside $Z$.
\end{lemma}

\newcommand{\ov}{\overline} 

\begin{proof}
Take an arbitrary arc $uv$ entering $Z$. Since $uv\notin M(Z)$, there
exists an $u\ov v$ set $F\in \cL$ so that $F\cap Z\ne
\emptyset$. Relocate the tail of this arc into a node in $F\cap
Z$. Let $D'$ be the digraph obtained from $D$ after doing this
tail-relocation for every arc entering $Z$. By Claim
\ref{cl:tailreloc}, there exists an \cL-tight arborescence in $D$ with
root outside $Z$ if and only if  there exists an \cL-tight
arborescence in $D'$ with root outside $Z$. But $\varrho_{D'}(Z)=0$.
\end{proof}

 \newcommand{\bD}{\bar D}

For any $F\in \cL$ and nonempty disjoint subsets
$Z_1,Z_2\subseteq F$ the set of arcs in $M_{D[F], \cL[F]}(Z_1)\cup
M_{D[F], \cL[F]}(Z_2) $ will be called an \emph{$\cL$-double cut}, and
we introduce the following notation for the minimum cardinality of an
$\cL$-double cut:
\begin{eqnarray*}
\Theta_F:=\Theta_{F,D}:=\Theta_{F,D,\cL}:=
\min\{f_{D[F],
  \cL[F]}(Z_1)+f_{D[F], \cL[F]}(Z_2): \  \emptyset \ne Z_1,Z_2\subseteq
F, Z_1\cap Z_2=\emptyset\}.
\end{eqnarray*}

The following simple observation is worth mentioning.
\begin{claimnum}\label{cl:Thetamu}
Given a digraph $D=(V,A)$ and a laminar family $\cL\subseteq 2^V$,
then $f_{D,\cL}(Z)\le \varrho_D(Z)$ holds for every $Z\subseteq
V$. Consequently, $\Theta_{F,D,\cL}\le \mu(D[F])$ holds for any $F\in
\cL$.
\end{claimnum}

Note that the tail-relocation operation introduced above does not
change the $f$-value of any set $Z\subseteq V$, that is
$f_{D,\cL}(Z)=f_{D',\cL}(Z)$, if $D'$ is obtained from $D$ by (one or
several) tail-relocation. Consequently, this operation does not modify
the $\Theta$ value, either, that is
$\Theta_{F,D,\cL}=\Theta_{F,D',\cL}$ for any $F\in \cL$. The
following "min-min" theorem motivates the definition of $\Theta$.

\begin{theorem}\label{thm:minimin}For a digraph $D=(V,A)$, and a laminar family
$\cL$ of subsets of $V$, the minimum number of arcs to be deleted from
  $D$ to obtain a digraph that does not contain an $\cL$-tight
  arborescence is attained on an $\cL$-double cut, that is
\[\gamma (D,\cL)=\min_{F\in \cL}\Theta_{F,D,\cL}.\]
\end{theorem}


\begin{proof}
By Lemma \ref{lem:norout}, $\gamma (D,\cL)\le \min_{F\in
  \cL}\Theta_F$, since if we delete an arc set
$M_{D[F]}(Z_1)\cup M_{D[F]}(Z_2)$ for some $F\in \cL$ and
non-empty disjoint $Z_1,Z_2\subseteq F$, then no $\cL[F]$-tight
arborescence survives in $D[F]$ (since its root can neither be in
$F-Z_1$, nor in $F-Z_2$, by Lemma \ref{lem:norout}, and $F=(F-Z_1)\cup (F-Z_2)$).

The other inequality follows from Lemma \ref{lem:minmin} below.
\end{proof}

\begin{lemma}\label{lem:minmin}
Given a digraph $D=(V, A)$ and a laminar family $\cL\subseteq 2^V$
with $V\in \cL$, if
\begin{equation}\label{eq:minmin}
f_{D[F]}(Z_1)+f_{D[F]}(Z_2)\ge 1\mbox{ for any }F\in \cL \mbox{
  and non-empty disjoint sets }Z_1,Z_2\subseteq F
\end{equation}
then there exists a
$\cL$-tight arborescence in $D$.
\end{lemma}

\begin{proof} Let $\bD=(V+\bar{r}, A\cup \{\bar{r}v: v\in V\})$. Notice that $\cL$-tight arborescences in $D$ and \cL-tight 
arborescences in $\bD$ naturally correspond to each other, since
$V\in\cL$.  Thus we will show that there exists an \cL-tight
arborescence in $\bD$ (notice that the condition holds for $\bD$,
since $\bD[F]=D[F]$ for every $F\in \cL$).  We will use induction on $|\cL|+|V|+|A(\bD)|$. If
$\cL=\{V\}$ then the lemma is true by Lemma \ref{lem:arb}. Otherwise
let $F\in \cL$ be an inclusionwise minimal member of $\cL$: again by
Lemma \ref{lem:arb}, there exists a spanning arborescence in
$\bD[F]$. Let $R$ be the subset of nodes of $F$ that can be the root
of a spanning arborescence in $\bD[F]$, i.e. $R=\{r\in F: $ there
exists an $r$-rooted arborescence (spanning $F$) in $\bD[F]\}$.
\begin{enumerate}
\item Assume first that $|R|\ge 2$ and let $\bD_1=\bD/R$ obtained by
  contracting $R$.  For any set $Z\subseteq V$ which is either
  disjoint form $R$, or contains $R$, let $Z/R$ be its (well-defined)
  image after the contraction and let $\cL_1=\{X/R: X\in \cL\}$.  By
  induction, there exists an $\cL_1$-tight arborescence $P$ in $\bD_1$,
  since $f_{\bD[X/R]}(Z/R)=f_{\bD[X]}(Z)$ for any $X/R\in \cL_1$ and
  $Z/R\subseteq X/R$. It is clear that we can create an $\cL$-tight
  arborescence in $\bD$ from $P$: we describe one possible
  way. Consider the unique arc in $P$ that enters $F$ and assume that
  the pre-image of this arc has head $r\in R$. Delete every arc from
  $P$ induced by $F/R$ and substitute them with an arbitrary
  $r$-rooted arborescence (spanning $F$) of $\bD[F]$. This clearly
  gives an $\cL$-tight arborescence.

\item So we can assume that $R=\{r\}$. Next assume that there exists
  an arc  $uv\in A(\bD)$ entering $F$ with $r\ne v$. Let
  $\bD_2=\bD-uv$: we claim that there exists an $\cL$-tight arborescence
  in $\bD_2$ (which is clearly an $\cL $-tight arborescence in $\bD$,
  too). If this does not hold then by the induction there must exist a
  set $F'\in \cL$ and non-empty disjoint subsets $Z_1,Z_2\subseteq F'$
  with $\sum_{i=1,2}f_{\bD_2[F']}(Z_i)=0$. Since
  $\sum_{i=1,2}f_{\bD[F']}(Z_i)>0$, the arc $uv$ must be equal to
  (say) $M_{\bD[F'],\cL}(Z_1)$ (while
  $M_{\bD[F'],\cL}(Z_2)=\emptyset$). This implies that $uv$ enters
  $Z_1$, while $r\in Z_1$ must also hold, otherwise
  $f_{\bD[F']}(Z_1)\ge 2$ would hold, since $v$ is reachable from $r$
  in $\bD[F']$. Let $Z_1'=Z_1-(F-r)$ and observe that
  $f_{\bD[F']}(Z_1')=0$: this is because the arcs in
  $\delta^{in}_{\bD[F']}(Z_1')-\delta^{in}_{\bD[F']}(Z_1)$ all leave
  $F$, since $\varrho_{\bD[F]}(r)=0$ by Lemma \ref{lem:roots}. Thus
  $f_{\bD[F']}(Z_1')+f_{\bD[F']}(Z_2)=0$, a contradiction.
\item Thus we can also assume that the arcs of $\bD$ entering $F$ all
  enter $r$. Let $\cL_2=\cL-\{F\}$: then clearly
  $f_{\bD[F'],\cL_2}(Z)\ge f_{\bD[F'],\cL}(Z)$ for any $F'\in \cL_2$
  and $Z\subseteq F'$, so by induction there exists an $\cL_2$-tight
  arborescence $P_2$ in $\bD$. Since the root of $P_2$ is $\bar{r}$ and $P_2$ enters $F$ only once (since all arcs entering $F$ have head $r$), $P_2$  is also $\cL$-tight,
  so the theorem is proved.
\end{enumerate}
\end{proof}

\subsection{A second proof of the ``min-min'' theorem}\label{sec:reprove}

In this subsection we reprove Lemma \ref{lem:minmin} which is the hard
direction of Theorem \ref{thm:minimin}. Our motivation is that we
think we can give  better insights into the details of our
techniques.

Let $\cL_0=\{\{v\}: v\in V\}\cup \{V\}$. Notice that an arborescence
is \cL-tight if and only if it is $(\cL\cup \cL_0)$-tight (and
$\cL\cup \cL_0$ is also laminar, if \cL\ is laminar). Therefore, in
Section \ref{sec:reprove} we will assume that $\cL_0\subseteq \cL$.
In what follows we abbreviate $f_{D}$ by $f$ (and $f_{D[F], \cL[F]}$
is abbreviated by $f_{D[F]}$ for an $F\in \cL$). For nodes $s,t\in V$,
a subset $S$ with $s\in S\subseteq V-t$ is called an $s\ov t$-set.

\begin{claimnum}\label{cl:stinM}
Given a set $Z\subseteq V$, an arc $st$ is in $M(Z)$ if and only if it
enters $Z$ and 
the largest $s\ov t$ set in $ \cL$ is disjoint from $Z$. 
\end{claimnum}

Notice that the largest $s\ov t$ set in $ \cL$ is well defined for
every pair $s, t\in V$, since $\cL_0\subseteq \cL$; in fact this
motivates the introduction of $\cL_0$.  


\begin{claimnum}\label{cl:Mcup}
For sets $X,Y\subseteq V$ we have $M(X\cup Y)\subseteq M(X)\cup
M(Y)$. Consequently,  $f(X)=f(Y)=0$ implies $f(X\cup Y)=0$.
\end{claimnum}

\begin{proof}
Consider an arc  $a\in M(X\cup Y)$: this must enter at least one of $X$
and $Y$, say it enters $X$. We claim that $a\in M(X)$, too: it cannot leave some $F\in \cL_X$,
since $\cL_X\subseteq \cL_{X\cup Y}$.
\end{proof}

\begin{claimnum}\label{cl:Mcap}
If  \eqref{eq:minmin} holds and $f(X)=f(Y)=0$ for some $X,Y\subseteq V$ 
then $f(X\cap Y)=0$, too.
  (Note that this does not necessarily hold without
  \eqref{eq:minmin}.)
\end{claimnum}

\begin{proof}
Assume $f(X\cap Y)>0$ and consider an
arc  $uv\in M(X\cap Y)$. Let $F\in \cL$ be the largest $u\ov v$ set
in \cL. Since $uv\in M(X\cap Y)$, $F\cap (X\cap Y)=\emptyset$, but
since $f(X)=f(Y)=0$, $F\cap X\ne\emptyset\ne F\cap Y$, that together with 
$f_{D[F]}(X\cap F)=f_{D[F]}(Y\cap F)=0$ contradicts \eqref{eq:minmin}.
\end{proof}

The following statement  extends parts of  Lemma \ref{lem:roots}.

\begin{claimnum}\label{cl:fR}
Given a digraph $D=(V,A)$ and laminar family $\cL\subseteq 2^V$, let
$R=\{r\in V:$ there exists   a $\cL$-tight arborescence in
$D$ with root $r\}$. Then $f(R)=0.$
\end{claimnum}
\begin{proof}
Assume $f(R)>0$ and let $uv\in M(R)$. Let $F$ be the largest $u\ov v$
set in \cL. We have $F\cap R=\emptyset$. Let $B\subseteq A$ be an
\cL-tight arborescence of root $v$, and let $xy\in B$ be the (unique)
arc entering $F$. Consider $B'=B+uv-xy$: this is an arborescence of
root $y$, since every node has in-degree 1, and $B'$ is a spanning
tree in the undirected sense. We claim that $B'$ is also $\cL$-tight
which contradicts the definition of $R$.
\end{proof}

\begin{proof}[Proof of Lemma \ref{lem:minmin}]
The proof is by induction on $|\cL|+|A|$. The base case $\cL=\cL_0$ is
clear. Assume $\cL\supsetneq \cL_0$ and let $F_0\in \cL$ be an inclusionwise maximal
member of $\cL-\cL_0$.  An arc in $uv\in \delta^{in}(F_0)$ is called \textbf{erasable}
if \eqref{eq:minmin} holds for \cL\ and $D-uv$, otherwise $uv$ is called
\textbf{essential}. If there exists an erasable arc then
we are done by induction.  So we can assume that every arc entering $F_0$
is essential.  We will prove that the number of arcs entering $F_0$ is at
most one.  In order to prove this, let us first investigate when an arc
is essential.  Let $R=\{r\in F_0: r$ can be the root of an $\cL[F_0]$-tight
arborescence in $D[F_0]\}$. By the induction hypothesis, $R\ne
\emptyset$.  An arc $uv$ entering $F_0$ is essential if and only if
there exist non-empty disjoint $Z_1,Z_2\subseteq V$ with
$M(Z_1)=\{uv\}$ and $M(Z_2)=\emptyset$ (here we used that $F_0$ is a maximal member of $\cL$): assume that this is the case.
Observe that $f_{D[F_0]}(Z_1\cap F_0)=0$, implying that
\begin{enumerate}
\item $F_0\cap Z_2=\emptyset$, otherwise $f_{D[F_0]}(Z_1\cap
  F_0)=f_{D[F_0]}(Z_2\cap F_0)=0$ would contradict \eqref{eq:minmin}, and
\item $R\subseteq Z_1$, by the definition of $R$ and Lemma \ref{lem:norout}.
\end{enumerate}
Next we prove that $f(Z_1- (F_0-R))\le f(Z_1)$. Let $\hat{Z_1}=Z_1-
(F_0-R)$: in fact we show $M(\hat{Z_1})\subseteq M(Z_1)$. Consider an arc
$st\in M(\hat{Z_1})$. If $s\notin F_0$ then the largest $s\ov t$ set
in $\cL$ must be disjoint from $F_0$, therefore $st\in M({Z_1})$,
too. So assume $s\in F_0$. Then $t$ cannot be in $F_0$ (that is, in
$R$), since that would imply $f_{D[F_0]}(R)>0$, contradicting Claim
\ref{cl:fR}. Therefore $t\notin F_0$, that is, $F_0$ is an $s\ov t$
member of \cL, contradicting $st\in M(\hat{Z_1})$ by Claim
\ref{cl:stinM} (note that $F_0\cap \hat{Z_1}$ equals the nonempty
set $R$).

The previous observation implies that $v\in R$ must hold
(otherwise $f(\hat{Z_1})+f(Z_2)=0$). So we can assume that $Z_1\cap
F_0=R$ (we can substitute $Z_1$ by $\hat{Z_1}$).  

Assume that there is
another arc $u'v'$ entering $F_0$. By our assumption, this arc is essential,
proven by the existence of non-empty disjoint $Z'_1,Z'_2\subseteq V$
with $M(Z'_1)=\{u'v'\}$ and $M(Z'_2)=\emptyset$. Notice that $Z_2\cap
Z'_2\ne \emptyset$ by \eqref{eq:minmin}. We can again assume that
$Z_1'\cap F_0=R$ and $v'\in R$. By Claim \ref{cl:Mcup}, $M(Z_1\cup
Z'_1)\subseteq \{uv,u'v'\}$. However, $M(Z_1\cup Z'_1)=\emptyset$, for
example $uv\notin M(Z_1\cup Z'_1)$, since $uv$ enters $Z_1\cap Z_1'$
and $uv\notin M(Z_1')$, either $uv$ does not enter $Z_1'$, or if it
enters, then the largest $u\ov v$ set in \cL\ intersects $Z_1'$. We
have a contradiction with \eqref{eq:minmin} having $f(Z_1\cup
Z_1')=f(Z_2\cap Z_2')=0$.

We have proved that $\varrho_D(F_0)\le 1$.  Let $\cL'=\cL-\{F_0\}$. The
conditions in \eqref{eq:minmin} clearly hold for $\cL'$ (since
$f_{D[F],\cL'}(Z)\ge f_{D[F],\cL}(Z)$ for every $F\in \cL'$
and $Z\subseteq F$), so by induction there exists an $\cL'$-tight
arborescence $B\subseteq A$. If $\varrho_D(F_0)=0$ then $B$ is also
\cL-tight. On the other hand, if $uv\in A$ enters $F_0$, then $uv$ was
essential (proven by non-empty disjoint $Z_1,Z_2\subseteq V$ with
$M(Z_1)=\{uv\}$ and $M(Z_2)=\emptyset$), but then $Z_2\cap
F_0=\emptyset$, as we observed earlier, implying (by Lemma
\ref{lem:norout}) that the root of $B$ can not be in $F_0$, so $B$ is
also \cL-tight, what was to be proven.
\end{proof}


\subsection{In-solid sets and anchor nodes}\label{sec:solid}

Theorem \ref{thm:minimin} suggests a strategy for solving Problem
\ref{prob:1}: for every $F\in \cL$ we find a minimum
$\cL[F]$-double-cut in $D[F]$ and output the best of these. For
finding the minimum $\cL[F]$-double-cut, tail relocation can help: if
$ \Theta_F=f_{D[F], \cL[F]}(Z_1)+f_{D[F], \cL[F]}(Z_2)$ for some
nonempty, disjoint pair $Z_1, Z_2\subseteq F$, then suitably
relocating the tails of all arcs that leave some $F'\in \cL[F]_{Z_i}$
to a node in $F'\cap Z_i$ we do not change $\Theta_F$ and
$f_{D[F]}(Z_i)$ becomes $\varrho_{D'[F]}(Z_i)$, thus the problem boils
down to finding a minimum double cut in $D'[F]$ (where $D'$ is obtained
from $D$ after some suitable tail-relocations). But we can not try
every possible tail-relocations, there are too many of those, we need
something more here.  These extra ideas are introduced in this
section.

\begin{definition}
A family of sets $\cF\subseteq 2^V$ of a finite ground set $V$ is said
to satisfy the Helly-property, if any sub-family $\cX$ of pairwise
intersecting members of $\cF$ has a non-empty intersection,
i.e. $\cX\subseteq \cF$ and $X\cap
X'\ne \emptyset$ for every $X,X'\in \cX$ implies that $\cap\cX\ne \emptyset$.
\end{definition}

The following definition is taken from \cite{bbf}.

\begin{definition}
Given a digraph $G=(V,A)$, a non-empty subset of nodes $X\subseteq V$
is called \emph{in-solid}, if $\varrho(Y)>\varrho(X)$ holds for
every nonempty $Y\subsetneq X$.
\end{definition}

\begin{theorem}[Bárász, Becker, Frank \cite{bbf}]
The family of in-solid sets  of a digraph satisfies the Helly-property.
\end{theorem}

The authors of \cite{bbf} prove in fact more: they show that the
family of in-solid sets is a \emph{subtree-hypergraph}, but we will
only use the Helly property here. The following theorem formulates the
key observation for the main result.

\begin{theorem}\label{thm:specnode}
In a digraph $G=(V,A)$ there exists a node $t\in V$ such that
$\varrho(Z)\ge \frac{\mu(G)}{2}$ for every non-empty $Z\subseteq
V-t$.
\end{theorem}

\begin{proof}
Consider the family $\cX=\{X\subseteq V: X$ is in-solid and
$\varrho(X)< \frac{\mu(G)}{2}\}$. If there were two disjoint members
$X,X'\in \cX$ then $\varrho(X)+\varrho(X')<\mu(G)$ would a contradict
the definition of $\mu(G)$. Therefore, by the Helly-property of the
in-solid sets, there exists a node $t\in \cap\cX$. This node satisfies
the requirements of the theorem, since if there was a non-empty
$Z\subseteq V-t$ with $\varrho(Z)< \frac{\mu(G)}{2}$, then $Z$ would
necessarily contain an in-solid set $Z'\subseteq Z$ with
$\varrho(Z')\le \varrho(Z)$ (this follows from the definition of in-solid
sets), contradicting the choice of $t$.
\end{proof}

In a digraph $G=(V,A)$, a node $a\in V$ with the property
$\varrho(Z)\ge \frac{\mu(G)}{2}$ for every non-empty $Z\subseteq
V-a$ will be called an \textbf{anchor node} of $G$.

\subsection{A polynomial-time algorithm}\label{sec:algorithm}

In this section we present a polynomial time algorithm to determine
the robustness of tight arborescences, which also implies a polynomial
time algorithm to determine the robustness of minimum cost
arborescences.  A sketch of the algorithm goes as follows. 
We maintain a subset $\cL'$ of $\cL$, which is initiated with $\cL':=\cL$. For a
minimal member $F$ of $\cL'$, we apply Theorem \ref{thm:specnode}, and
find an anchor node $a_F$ of $D[F]$. We replace the tail of every arc leaving
$F$ by $a_F$, remove $F$ from $\cL'$, and repeat until $\cL'$ goes
empty. This way we construct a sequence of digraphs on the same node
set: let $D'$ be the last member of this sequence. Then for any $a\in V$ we construct another digraph $D_a$ from $D'$: for every
$F\in\cL$ with $a \in F$ and every arc of $D'$ leaving $F$ we replace
the tail of this arc with $a$. Finally, we determine minimum double cuts
in $D'[F]$ for every $F\in \cL$, and we also determine
minimum double cuts in $D_a[F]$ for every $F\in \cL$ with
$a\in F$: this way
we have determined $O(n^2)$ double cuts altogether. Each of these
double cuts also determines an $\cL$-double cut in $D$, and we pick the
one with the smallest cardinality, to claim that it actually is
optimal.
The algorithm is given as a pseudocode below.


\newcommand{\covalg}{COVERING\_\-TIGHT\_\-ARBORESCENCES}

\newcommand{\opt}{best}

\begin{pszkod}
{Algorithm \covalg}

\item[] \textbf{INPUT} A digraph $D=(V,A)$ and a laminar family
  $\cL\subseteq 2^V$ with $V\in \cL$

\item[] \textbf{OUTPUT} $\gamma(D,\cL)$

\item[] /* First Phase: creating graphs $D'$ and $D_a$ */

\item Let  $D'=D$ and $\cL'=\cL$.

\item While $\cL'\ne  \emptyset$ do

\tab

  \item Choose an inclusionwise minimal set $F\in \cL'$

  \item Let $a_F\in F$ be an anchor node of $D'[F]$ (apply Theorem \ref{thm:specnode} to $G=D'[F]$)\label{st:anchor}
  
  \item Modify $D'$: relocate the tail of all arcs leaving $F$ to $a_F$

  \item Let  $\cL'=\cL'-F$

\untab

\item For every $a\in V$ 
\tab

  \item Let $D_a$ be obtained from $D'$ by relocating the tail of every arc
    leaving a set $F\in \cL$ with $a\in F$ to $a$

\untab

\item[] /*Second Phase: finding the optimum*/

\item Let $\opt=+\infty$ and $\cL'=\cL$.

\item While $\cL'\ne  \emptyset$ do
\tab

  \item Choose an inclusionwise minimal set $F\in \cL'$

  \item If $\opt>\mu(D'[F])$ then $\opt:=\mu(D'[F])$\label{st:muD}

  \item For each $a\in F$ do

  \tab
  

    \item If $\opt>\mu(D_{a}[F])$ then $\opt:=\mu(D_{a}[F])$\label{st:muDr}

  \untab

  \item Let  $\cL'=\cL'-F$

\untab

\item Return $\opt$.

\end{pszkod}







  
  








The algorithm above is formulated in a way that it returns the optimum
$\gamma(D,\cL)$ in question, but by the correspondance between the arc
set of $D$ and that of $D'$ and $D_a$ in the algorithm, clearly we can
also return the optimal arc set, too. It is also clear that the
algorithm can be formulated to run in strongly polynomial time for
Problem \ref{prob:1'}, too: we only need to modify the definition of
the in-degree function $\varrho_D$, so that the weights are taken into
account.

\begin{theorem}
The Algorithm \textsf{\covalg} returns a correct answer.
\end{theorem}

\begin{proof}
First of all, since $\Theta_{F,D}=\Theta_{F,D'}=\Theta_{F,D_a}$ for
any $F\in \cL$ and $a\in F$, and $\Theta_{F,D'}\le
\mu(D'[F])$ and $\Theta_{F,D_a}\le \mu(D_a[F])$, the algorithm returns
an upper bound for the optimum $\gamma(D,\cL)$ in question by Theorem
\ref{thm:minimin}.

On the other hand, assume that $F$ is an inclusionwise minimal member
of $\cL$ such that the optimum $\gamma(D,\cL)=\Theta_{F,D}$
(such a set exists again by Theorem \ref{thm:minimin}). Assume
furthermore that the non-empty disjoint sets $Z_1,Z_2\subseteq F$ are
such that $\Theta_{F,D}=f_{D[F]}(Z_1)+f_{D[F]}(Z_2)$. The following
sequence of observations proves the theorem.
\begin{enumerate}
\item First observe, that any member $F'\in \cL$ which is a proper
  subset of $F$ can intersect at most one of $Z_1$ and $Z_2$. Assume
  the contrary, and note that $f_{D[F]}(Z_i)\ge f_{D[F']}(Z_i\cap F')$
  holds for $i=1,2$, contradicting the minimal choice of $F$.
\item Next observe that there do not exist two disjoint members
  $F',F''\in \cL_{Z_1\cup Z_2}$ that are proper subsets of $F$ such
  that $a_{F'}$ and $a_{F''}$ are both outside $Z_1\cup Z_2$. To see
  this assume again the contrary and let $F',F''$ be two inclusionwise
  minimal such sets. By exchanging the roles of $Z_1$ and $Z_2$ or the
  roles of $F'$ and $F''$ we arrive at the following two cases: either
  both $F'$ and $F''$ intersect $Z_1$, or $F'$ intersects $Z_1$ and
  $F''$ intersects $Z_2$. The proof is analogous for both cases.
  Assume first  that both $F'$ and $F''$
  intersect $Z_1$. Then we have
\begin{multline}
\gamma(D,\cL)=\Theta_{F,D}=\Theta_{F,D'}=f_{D'[F]}(Z_1)+f_{D'[F]}(Z_2)\ge \\\ge
f_{D'[F]}(Z_1)\ge  f_{D'[F']}(Z_1\cap F')+f_{D'[F'']}(Z_1\cap
F'')=\\
=\varrho_{D'[F']}(Z_1\cap F')+\varrho_{D'[F'']}(Z_1\cap F'')\ge
\\\ge \frac{\mu({D'[F']})}{2} + \frac{\mu({D'[F'']})}{2} 
>\gamma(D,\cL),
\end{multline}
a contradiction. 
Here the second inequality follows from
the definition of the function $f$, the equality following it is
because $a_{F'''}\in Z_1$ if $F'''\in \cL_{Z_1}$ is a proper subset of
$F'$ or $F''$. The next inequality follows from the definition of
$a_{F'}$ and $a_{F''}$, and the last (strict) inequality is by the
minimal choice of $F$.

In the other case, when  $F'$ intersects $Z_1$ and
  $F''$ intersects $Z_2$, we get  the contradiction in a similar way:
\begin{multline}
\gamma(D,\cL)=\Theta_{F,D}=\Theta_{F,D'}=f_{D'[F]}(Z_1)+f_{D'[F]}(Z_2)\ge \\
\ge  f_{D'[F']}(Z_1\cap F')+f_{D'[F'']}(Z_2\cap F'')=
\varrho_{D'[F']}(Z_1\cap F')+\varrho_{D'[F'']}(Z_2\cap F'')\ge
\\\ge \frac{\mu({D'[F']})}{2} + \frac{\mu({D'[F'']})}{2} 
>\gamma(D,\cL).
\end{multline}

\item Therefore we are left with two cases. In the first case assume
  that $a_{F'}\in Z_1\cup Z_2$ holds for every $F'\in \cL_{Z_1\cup Z_2}$
  with $F'\subsetneq F$. In that case we have that
  $f_{D'[F]}(Z_i)=\varrho_{D'[F]}(Z_i)$ for both $i=1,2$, and thus
  $\gamma(D,\cL)=\Theta_{F,D'}=\sum_{i=1,2}\varrho_{D'[F]}(Z_i)\ge
  \mu(D'[F])\ge \Theta_{F,D'}$.

\item In our last case there exists a unique inclusionwise minimal
  $F'\in \cL_{Z_1\cup Z_2}$ such that $F'$ is a proper subset of $F$
  with $a_{F'}\notin Z_1\cup Z_2$. Assume without loss of generality that
  $F'$ intersects ${Z_1}$ and choose an arbitrary $a\in F'\cap
  Z_1$. Then $f_{D_a[F]}(Z_i)=\varrho_{D_a[F]}(Z_i)$ for both $i=1,2$,
  and thus
  $\gamma(D,\cL)=\Theta_{F,D_a}=\sum_{i=1,2}\varrho_{D_a[F]}(Z_i)\ge
  \mu(D_a[F])\ge \Theta_{F,D_a}$.

\end{enumerate}
\end{proof}


\subsection{Running time}\label{sec:runtime}

Let $T(N,M)$ be the time needed to find a minimum $s-t$ cut in an
edge-weighted digraph having $N$ nodes and $M$ arcs (that is, $M\le
N^2$ here). 




The natural weighted version of Problem \ref{prob:3} is the following.
\begin{problem}\label{prob:3w}
Given a digraph $D=(V,A)$, and a nonnegative weight function $w:A\to
\Rset_+$, and a laminar family $\cL\subseteq 2^V$, find a subset $H$
of the arc set such that $H$ intersects every $\cL$-tight arborescence
and $w(H)$ is minimum.
\end{problem}

As mentioned above, if we want to solve the weighted Problem
\ref{prob:3w}, the only thing to be changed is that the in-degree
$\varrho(X)$ of a set should  mean the weighted in-degree. We will
analyze the algorithm in this sense, so we assume that the input
digraph does not contain parallel arcs, but weighted ones.

In order to analyze the performance of Algorithm \covalg, let $n$ and
$m$ denote the number of nodes and arcs in its input (so $m\le n^2$).

To implement the algorithm above we need 2 subroutines.  The first
subroutine finds an anchor node in an edge-weighted digraph.  This
subroutine will be used $|\cL|\le n$ times in Step \ref{st:anchor} for
digraphs having at most $n$ nodes and at most $m$ arcs.  By the
definition of anchor nodes, any node $r$ maximizing
$\min\{\varrho_G(X):\emptyset\ne X\subseteq V-r\}$ can serve as an
anchor node. Therefore, finding an anchor can be done in $n^2T(n,m)$ time .

The second subroutine determines $\mu(G)$ for a given edge-weighted
digraph $G$.  This subroutine is used at most $n$ times in Step
\ref{st:muD} and $n^2$ times in Step \ref{st:muDr} for digraphs having
at most $n$ nodes and at most $m$ arcs. Note however that these
suboutine calls are not independent from each other, and we will make
use of this fact later.


We can determine $\mu(G)$ for a given edge-weighted digraph $G$ the
following way. Take two disjoint copies of $G$, and reverse all arcs
in the first copy (and denote this modified first copy by $G^1$). Let
the second copy be denoted by $G^2$, and for each $v\in V(G)$ let the
corresponding node in $V(G^i)$ be $v^i$ for $i=1,2$.  For each $v\in
V(G)$ add an arc $v^1v^2$ of infinite capacity from $v^1$ to its
corresponding copy $v^2\in V(G^2)$. This way we define an auxiliary
weighted digraph $\hat G=(V(G^1)\cup V(G^2), A(G^1)\cup A(G^2)\cup \{v^1v^2: v\in V(G)\})$. It is easy to see that for some $s\ne t$ nodes in
$V(G)$ we have $\min\{\delta_{\hat G}(Z):s^1\in Z\subseteq V(\hat
G)-t^2\}=\min\{\varrho_G(X)+\varrho_G(Y): s\in X\subsetneq V(G), t\in
Y\subsetneq V(G), X\cap Y=\emptyset\}$. Thus, by trying every possible
pair $s,t$, we can calculate $\mu(G)$ with $n^2$ minimum $s^1-t^2$-cut
computations in $\hat G$ in time $n^2T(n,m)$. 

We will calculate $\mu(G)$ for $O(n^2)$ graphs $G$ (each having at
most $n$ nodes and $m$ arcs): $n$ times in Step \ref{st:muD} for the
graphs $D'[F]$, and $n^2$ times in Step \ref{st:muDr} for the graphs
$D_{a}[F]$. On the other hand, as mentioned earlier, these calls are
not independent from each other, since if $\mu(D_{a}[F])<\mu(D'[F])$
for some $F\in \cL$ and $a\in F$ then $a\in X\cup Y$ has to hold for
the (optimal disjoint non-empty) sets $X,Y\subseteq F$ giving
$\mu(D_{a}[F])=\varrho_{D_{a}[F]}(X)+\varrho_{D_{a}[F]}(Y)$. Therefore
checking whether $\mu(D_{a}[F])<\opt$ or not in Step \ref{st:muDr}, we
only need to calculate  minimum $a^1-t^2$-cuts in $\widehat {D_a[F]}$
(for the node $a^1\in V((D_{a}[F])^1)$ corresponding to $a$ and every
$t^2\in V((D_{a}[F])^2)$ corresponding to nodes $t\in F-a$), needing only
$n$ minimum cut computations.

Putting everything together we get that the Algorithm \covalg\ can be
implemented to run in $n^3T(n,m)$ time.
It seems possible to further reduce the complexity of Steps
\ref{st:anchor} and \ref{st:muD}, however we don't know how to do this for Step
\ref{st:muDr}.


\subsection{Remarks on the polyhedral approach}
The tractability of the weighted Problem 5 is equivalent with optimization over the following polyhedron: 
\begin{align*}
\mathcal{P}:=conv(\{\chi _H:H\text{ a $\cL$-double cut}\})+\mathbb{R}_+^A.
\end{align*}
A completely different approach to the problem would be to directly show that this polyhedron $\mathcal{P}$ is tractable, which hinges upon finding a nice polyhedral description of the given polyhedron. Firstly, the polyhedron has facets with large coefficients, which rules out a rank-inequality type description. Secondly, the polyhedron seems to be of a composite nature in the following sense. For $\cL =\emptyset$,  
\begin{align*}
\mathcal{P}=\textstyle conv \left( \bigcup _{s\ne t, s,t\in V} conv(\{\chi _H:H\text{ a double cut separating $s, t$}\})+\mathbb{R}_+^A \right) ,
\end{align*}
where a double cut is said to separate $s,t$ if $s\in Z_1, t\in
Z_2$. Thus optimization over $\mathcal{P}$ reduces to optimization
over ${n\choose 2}$ polyhedra of double cuts separating a given pair
$s,t$. For any given pair $s,t$, this polyhedron has a nice
description, and also nice combinatorial algorithm for
optimization. When we apply this approach to a general $\cL$, then we
need to consider the union of an exponential number of polyhedra: one
for every possible choice of an anchor node in every set of
$\cL$. Thus the proposed approach only results in an efficient
algorithm for the special case $\cL =\emptyset$, and leaves the
general case without a polyhedral description.

\section{Further notes}\label{sec:fur}

As we have already pointed out, Problems \ref{prob:1} and
\ref{prob:1'} can be interpreted as covering minimum cost common bases
of two matroids.  In fact they can be interpreted as covering common
bases of a graphic matroid and a partition matroid, too, as we show
below. Let us start with a general observation which states that the
problem of covering minimum cost common bases of two matroids can be
reduced to the problem of covering common bases of two matroids.

\begin{theorem}\label{thm:commonbase}
Given two matroids $M_1$ and $M_2$ over the same ground set $S$, a
cost function $c:S\to \Rset$,  assume that $M_1$ and $M_2$ have a
common base, then there exist matroids $M_1'$ and $M_2'$ over $S$ so
that the minimum $c$-cost common bases of $M_1$ and $M_2$ are the
common bases of $M_1'$ and $M_2'$.
\end{theorem}

\begin{proof}
The simplest way to see this is through Frank's Weight Splitting
Theorem \cite{frank1981weighted}. It states, that the assumptions of our
theorem imply the existence of cost functions $c_1, c_2:S\to \Rset$ so
that $c=c_1+c_2$ and a common base of $M_1$ and $M_2$ has minimal
$c$-cost if and only if it is a $c_i$-minimum cost base of $M_i$ for
both $i=1,2$. As the family $\cB_i=\{B\subseteq S: B $ is a
$c_i$-minimum cost base of $M_i\}$ is the family of bases of a matroid
$M_i'$, the theorem follows.
\end{proof}

Let us specialize Theorem \ref{thm:commonbase} for our problems. Given
a digraph $D=(V,A)$ with a designated node $r\in V$, and a cost
function $c:A\to \Rset$, as in Problem \ref{prob:1}, Fukerson's
Theorem (Theorem \ref{thm:fulk}) tells us that we can find in polynomial time a
subset $A'\subseteq A$ of arcs (\emph{tight arcs}) and a laminar
family $\cL\subseteq 2^{V-r}$ such that an $r$-arborescence is of
minimum cost if and only if it uses only tight arcs and it is
$\cL$-tight. Using Claim \ref{cl:Ltight}  we
get the following:  $\{B\subseteq A: B$ is an optimal
$r$-arborescence$\} = \{B\subseteq A': B[F]$ is a tree spanning $F$
for every $F\in \cL\cup \{V\}$ and $\varrho_B(v)=1$ if $v\in V-r$ and $\varrho_B(r)=0\}$.
Let us introduce the definition of $\cL$-tight spanning trees.

\begin{definition}
Given a connected graph $G=(V,E)$ and a family $\cF\subseteq 2^V$, a
spanning tree $B\subseteq E$ is said to be \textbf{\cF-tight} if
$B[F]$ is a tree spanning $F$ for every $F\in \cF$.
\end{definition}

The following claim is easy to prove.

\begin{claimnum}
\label{cl:tightmaxcost}
Given a graph $G=(V,E)$ and a family $\cF\subseteq 2^V$, define a cost
function $c_1:E\to \Zset_+$ as $c_1(e)=|\{F\in \cF: e\subseteq F\}|$. If
there exists a \cF-tight spanning tree then the \cF-tight spanning trees are the maximum $c_1$-cost spanning trees.
\end{claimnum}

We will need one more observation stating that the matroid of maximum cost spanning trees of a graph is a graphic matroid.

\begin{claimnum}\label{cl:maxsptgraphic}
Given a connected graph $G=(V,E)$ and a cost function $c_1:E\to
\Rset$, there exists another graph $G'=(V', E')$ and a bijection
$\phi:E\to E'$ so that $T\subseteq E$ is a maximum $c_1$-cost spanning
tree in $G$ if and only if $\phi(T)$ is an inclusionwise maximal
forest in $G'$ (that is, a base of the circuit matroid of $G'$).
\end{claimnum}

Applying Claims \ref{cl:tightmaxcost} and \ref{cl:maxsptgraphic} to Problem \ref{prob:1} we get the following theorem.

\begin{theorem}
Given a digraph $D=(V,A)$ with a designated node $r\in V$, and a cost
function $c:A\to \Rset$, we can find in polynomial time a subset
$A'\subseteq A$, a graphic matroid $M_1$ and a partition matroid $M_2$
(both on ground set $A'$) so that $B\subseteq A$ is a minimum $c$-cost $r$-arborescence of $D$ if and only if $B\subseteq A'$ and it is a common base of $M_1$ and $M_2$.
\end{theorem}

Thus, the problem of covering optimal $r$-arborescences (Problem
\ref{prob:1}) can be reduced to the problem of covering common bases of
a graphic matroid and a partition matroid (with some correlation
between these two matroids).  Following the reformulations above one can
probably work out a polynomial time algorithm (different from ours)
for Problem \ref{prob:1}, too. This can be  a direction for further research.


\subsubsection*{Acknowledgements}
We thank Naoyuki Kamiyama for calling our attention to this problem at
the 7th Hungarian-Japanese Symposium on Discrete Mathematics and Its
Applications in Kyoto.  We would like to thank Krist\'of B\'erczi,
András Frank, Erika Kov\'acs, Tam\'as Kir\'aly and Zolt\'an Kir\'aly
of the Egerváry Research Group for useful discussions and remarks.  A
version of this paper was presented at the 21st International
Symposium on Mathematical Programming (ISMP 2012) and at the 16th
Conference on Integer Programming and Combinatorial Optimization (IPCO
2013).

\bibliographystyle{amsplain} \bibliography{bmincostarb}

\end{document}